\documentclass [a4 paper,12pt]{article}
\usepackage{mathrsfs}

\usepackage[usenames]{color}
\usepackage{bbm,amsmath,amssymb,amsfonts,amsthm,graphics,graphicx}

\newtheorem{lem}{Lemma}[section]
\newtheorem{thm}[lem]{Theorem}

\DeclareMathOperator{\V}{\mathbbm{Var}}
\DeclareMathOperator{\E}{\mathbbm{E}}
\DeclareMathOperator{\En}{\mathscr{E}}

\DeclareMathOperator{\si}{\sigma}
\DeclareMathOperator{\la}{\lambda}

 \DeclareMathOperator{\Ep}{Ep}
\DeclareMathOperator{\Ek}{Ek}

\title{The energy of random graphs
\footnote{Supported by NSFC No.10831001, PCSIRT and the ``973"
program.}}

\author{ \small Wenxue Du,~~Xueliang Li,~~Yiyang Li \\
\small Center for Combinatorics and LPMC-TJKLC\\
\small Nankai University, Tianjin 300071, P.R. China\\
\small Email: lxl@nankai.edu.cn}

\date{}

\begin{document}
\maketitle

\begin{abstract}
In 1970s, Gutman introduced the concept of the energy $\En(G)$ for a
simple graph $G$, which is defined as the sum of the absolute values
of the eigenvalues of $G$. This graph invariant has attracted much
attention, and many lower and upper bounds have been established for
some classes of graphs among which bipartite graphs are of
particular interest. But there are only a few graphs attaining the
equalities of those bounds. We however obtain an exact estimate of
the energy for almost all graphs by Wigner's semi-circle law, which
generalizes a result of Nikiforov. We further investigate the energy
of random multipartite graphs by considering a generalization of
Wigner matrix, and obtain some estimates of the energy for random
multipartite graphs.\\[3mm]

\end{abstract}

\section{Introduction}

Throughout this paper, $G$ denotes a simple graph of order $n$. The
eigenvalues $\lambda_1,\ldots,\lambda_n$ of the adjacency matrix
$\mathbf{A}(G)=(a_{ij})_{n\times n}$ of $G$ are said to be the {\em
eigenvalues of the graph $G$}. In chemistry, the eigenvalues of a
molecular graph has a closed relation to the molecular orbital
energy levels of $\pi$-electrons in conjugated hydrocarbons. For the
H\"{u}chkel molecular orbital approximation, the total
$\pi$-electron energy in conjugated hydrocarbons is given by the sum
of absolute values of the eigenvalues of the corresponding molecular
graph in which the maximum degree is not more than 4 in general. In
1970s, Gutman \cite{gut} extended the concept of energy $\En(G)$ to
all simple graphs $G$, and defined that
$$\En(G)=\sum_{i=1}^n|\lambda_i|,$$
where $\lambda_1,\ldots,\lambda_n$ are the eigenvalues of $G$.
Evidently, one can immediately get the energy of a graph by
computing the eigenvalues of the graph. It is rather hard, however,
to compute the eigenvalues for a large matrix, even for a large
symmetric (0,1)-matrix like $\mathbf{A}(G)$. So many researchers
established a lot of lower and upper bounds to estimate the
invariant for some classes of graphs among which the bipartite
graphs are of particular interest. For further details, we refer the
readers to the comprehensive survey \cite{glz}. But there is a
common flaw for those inequalities that only a few graphs attain the
equalities of those bounds. Thus we can hardly see the major
behavior of the invariant $\En(G)$ for most graphs with respect to
other graph parameters ($|V(G)|$, for instance). In this paper,
however, we shall present an exact estimate of the energy for almost
all graphs by Wigner's semi-circle law. Moreover, we investigate the
energy of random multipartite graphs by employing the results on the
spectral distribution of band matrix which is a generalization of
Wigner matrix.

The structure of our article is as follows. In the next section,
we shall consider the random graphs constructed from the
classical Erd\"{o}s--R\'{e}nyi model. The second model is
concerned with random multipartite graphs which will be defined
and explored in the last section.

\section{The energy of $G_n(p)$}

In this section, we shall formulate an exact estimate of the energy
for almost all graphs by Wigner's semi-circle law.

We start by recalling the Erd\"{o}s--R\'{e}nyi model
$\mathcal{G}_{n}(p)$ (see \cite{BB}), which consists of all
graphs with vertex set $[n]=\{1,2,\ldots,n\}$ in which the edges
are chosen independently with probability $p=p(n)$. Apparently,
the adjacency matrix $\mathbf{A}(G_n(p))$ of the random graph
$G_n(p)\in\mathcal{G}_{n}(p)$ is a random matrix, and thus one
can readily evaluate the energy of $G_n(p)$ once the spectral
distribution of the random matrix $\mathbf{A}(G_n(p))$ is known.

In fact, the study on the spectral distributions of random matrices
is rather abundant and active, which can be traced back to
\cite{Wi}. We refer the readers to \cite{B, De, Me} for an overview
and some spectacular progress in this field. One important
achievement in that field is Wigner's semi-circle law which
characterizes the limiting spectral distribution of the empirical
spectral distribution of eigenvalues for a sort of random matrix.

In order to characterize the statistical properties of the wave
functions of quantum mechanical systems, Wigner in 1950s
investigated the spectral distribution for a sort of random matrix,
so-called {\em Wigner matrix},
$$\mathbf{X}_n:=(x_{ij}), ~~1\le i,j\le n,$$
which satisfies the following properties:\begin{itemize}

\item $x_{ij}$'s are independent random variables with $x_{ij}=x_{ji}$;

\item the $x_{ii}$'s have the same distribution $F_1$,
while the $x_{ij}$'s $(i\neq j)$ have the same distribution $F_2$;

\item $\V(x_{ij})=\si_2^2<\infty$ for all $1\le i< j\le n$.
\end{itemize}
We denote the eigenvalues of $\mathbf{X}_n$ by
$\la_{1,n},\la_{2,n},\ldots,\la_{n,n}$, and their empirical spectral
distribution (ESD) by
$$\Phi_{\mathbf{X}_n}(x)=
\frac{1}{n}\cdot\#\{\la_{i,n}\mid\la_{i,n}\le x, ~i=1,2,\ldots,
n\}.$$
Wigner \cite{W55,W58} considered the limiting spectral
distribution (LSD) of $\mathbf{X}_n$, and obtained his semi-circle
law.
\begin{thm}\label{Thm-1} Let $\mathbf{X}_n$ be a Wigner matrix. Then
$$\lim_{n\rightarrow\infty}\Phi_{n^{-1/2}\mathbf{X}_n}(x)=\Phi(x)
\mbox{ a.s. }$${\it i.e.,} with probability 1, the ESD
$\Phi_{n^{-1/2}\mathbf{X}_n}(x)$ converges weakly to a distribution
$\Phi(x)$ as $n$ tends to infinity, where $\Phi(x)$ has the density
$$\phi(x)=\frac{1}{2\pi\si_2^2}\sqrt{4\si_2^2- x^2
}~\mathbf{1}_{|x|\le2\si_2}.$$
\end{thm}

\noindent{\bf Remark 2.1.} One of classical methods to prove the
theorem above is the moment approach. Employing the method, we can
get more information about the LSD of Wigner matrix. Set $\mu_i=\int
x\hspace{2pt}dF_i$ $(i=1,2)$ and $\overline{\mathbf{X}}_n
=\mathbf{X}_n-\mu_1\mathbf{I}_n-\mu_2(\mathbf{J}_n-\mathbf{I}_n)$,
where $\mathbf{I}_n$ is the unit matrix of order $n$ and
$\mathbf{J}_n$ is the matrix of order $n$ in which all entries equal
1. It is easily seen that the random matrix
$\overline{\mathbf{X}}_n$ is a Wigner matrix as well. By means of
Theorem \ref{Thm-1}, we have
\begin{equation}\label{Equ-3}
\lim_{n\rightarrow\infty}\Phi_{n^{-1/2}\overline{\mathbf{X}}_n}(x)=\Phi(x)\mbox{
a.s.}\end{equation} Evidently, each entry of
$\overline{\mathbf{X}}_n$ has mean 0. Furthermore, one can show,
using moment approach, that for each positive integer $k$,
\begin{equation}\label{Equ-1}
\lim_{n\rightarrow\infty}\int
x^kd\Phi_{n^{-1/2}\overline{\mathbf{X}}_n}(x)=\int
x^kd\Phi(x)\mbox{ a.s.}
\end{equation}
It is interesting that the existence of the second moment of the
off-diagonal entries is the necessary and sufficient condition for
the semi-circle law, but there is no moment requirement on the
diagonal elements. For further comments on the moment approach and
Wigner's semi-circle law, we refer the readers to the extraordinary
survey by Bai \cite{B}.\vspace{6pt}

We shall say that {\em almost every (a.e.)} graph in
$\mathcal{G}_n(p)$ has a certain property $Q$ (see \cite{BB}) if the
probability that a random graph $G_n(p)$ has the property $Q$
converges to 1 as $n$ tends to infinity. Occasionally, we shall
write {\em almost all} instead of almost every. It is easy to see
that if $F_1$ is a {\em pointmass at 0}, {\it i.e.,} $F_1(x)=1$ for
$x\ge 0$ and $F_1(x)=0$ for $x<0$, and $F_2$ is the {\em Bernoulli
distribution with mean $p$}, then the Wigner matrix $\mathbf{X}_n$
coincides with the adjacency matrix $\mathbf{A}(G_n(p))$ of the
random graph $G_n(p)$. Obviously, $\si_2=\sqrt{p(1-p)}$ in this
case.

To establish the exact estimate of the energy $\En(G_n(p))$ for a.e.
graph $G_n(p)$, we first present some notions. In what follows, we
shall use $\mathbf{A}$ to denote the adjacency matrix
$\mathbf{A}(G_n(p))$ for convenience. Set
$$\overline{\mathbf{A}}=\mathbf{A}-p(\mathbf{J}_n-\mathbf{I}_n).$$
It is easy to check that each entry of $\overline{\mathbf{A}}$ has
mean 0. We define the {\em energy $\En(\mathbf{M})$ of a matrix
$\mathbf{M}$} as the sum of absolute values of the eigenvalues of
$\mathbf{M}$. By virtue of the following two lemmas, we shall
formulate an estimate of the energy $\En(\overline{\mathbf{A}})$,
and then establish the exact estimate of
$\En(\mathbf{A})=\En(G_n(p))$ by using Lemma \ref{ky}.

Let $I$ be the interval $[-1,1]$.
\begin{lem}\label{Lem-1}
Let $I^c$ be the set $\mathbb{R}\setminus I$. Then
$$\lim_{n\rightarrow\infty}\int_{I^c}
x^2d\Phi_{n^{-1/2}\overline{\mathbf{A}}}(x)= \int_{I^c}
x^2d\Phi(x)\mbox{ a.s.}$$
\end{lem}
\begin{proof}[\bf Proof] Suppose
$\phi_{n^{-1/2}\overline{\mathbf{A}}}(x)$ is the density of
$\Phi_{n^{-1/2}\overline{\mathbf{A}}}(x)$. According to
Eq.(\ref{Equ-3}), with probability 1,
$\phi_{n^{-1/2}\overline{\mathbf{A}}}(x)$ converges to $\phi(x)$
almost everywhere as $n$ tends to infinity. Since $\phi(x)$ is
bounded on $I$, it follows that with probability 1,
$x^2\phi_{n^{-1/2}\overline{\mathbf{A}}}(x)$ is bounded almost
everywhere on $I$. Invoking bounded convergence theorem yields
$$\lim_{n\rightarrow\infty}\int_{I}
x^2d\Phi_{n^{-1/2}\overline{\mathbf{A}}}(x)= \int_{I}
x^2d\Phi(x)\mbox{ a.s.}$$

Combining the above fact with Eq.(\ref{Equ-1}), we have
\begin{eqnarray*}
\lim_{n\rightarrow\infty}\int_{I^c}
x^2d\Phi_{n^{-1/2}\overline{\mathbf{A}}}(x) &=&
\lim_{n\rightarrow\infty}\left(\int
x^2d\Phi_{n^{-1/2}\overline{\mathbf{A}}}(x)-\int_{I}
x^2d\Phi_{n^{-1/2}\overline{\mathbf{A}}}(x)\right)\\
&=&\lim_{n\rightarrow\infty}\int
x^2d\Phi_{n^{-1/2}\overline{\mathbf{A}}}(x)-
\lim_{n\rightarrow\infty}\int_{I}
x^2d\Phi_{n^{-1/2}\overline{\mathbf{A}}}(x) \\
&=& \int x^2d\Phi(x)-\int_{I} x^2d\Phi(x)\mbox{ a.s.}\\
&=&\int_{I^c} x^2d\Phi(x)\mbox{ a.s.}\end{eqnarray*}
\end{proof}

\begin{lem}[\bf Billingsley \cite{Bil} pp. 219]\label{Billinsley}
Let $\mu$ be a measure. Suppose that functions $a_n, b_n, f_n$
converges almost everywhere to functions $a, b, f$, respectively,
and that $a_n\le f_n\le b_n$ almost everywhere. If $\int
a_nd\mu\rightarrow\int a\hspace{2pt}d\mu$ and $\int
b_nd\mu\rightarrow\int b\hspace{2pt}d\mu$, then $\int
f_nd\mu\rightarrow\int fd\mu$.\end{lem}

We now turn to the estimate of the energy
$\En(\overline{\mathbf{A}})$. To this end, we first investigate the
convergence of $\int |x|d\Phi_{n^{-1/2}\overline{\mathbf{A}}}(x)$.
According to Eq.(\ref{Equ-3}) and the bounded convergence theorem,
we can deduce, by an argument similar to the first part of the proof
of Lemma \ref{Lem-1}, that
$$\lim_{n\rightarrow\infty} \int_{I}
|x|d\Phi_{n^{-1/2}\overline{\mathbf{A}}}(x)=\int_{I}|x|d\Phi(x)\mbox{
a.s.}$$ Obviously, $|x|\le x^2$ if $x\in I^c:=\mathbb{R}\setminus
I$. Set $a_n(x)=0,
b_n(x)=x^2\phi_{n^{-1/2}\overline{\mathbf{A}}}(x),$ and
$f_n(x)=|x|\phi_{n^{-1/2}\overline{\mathbf{A}}}(x)$. Employing
Lemmas \ref{Lem-1} and \ref{Billinsley}, we have
$$\lim_{n\rightarrow\infty}\int_{I^c}
|x|d\Phi_{n^{-1/2}\overline{\mathbf{A}}}(x)=\int_{I^c}|x|d\Phi(x)\mbox{
a.s.}$$ Consequently,
\begin{equation}\label{Convergence}
\lim_{n\rightarrow\infty}\int
|x|d\Phi_{n^{-1/2}\overline{\mathbf{A}}}(x) =\int|x|d\Phi(x)\mbox{
a.s.}\end{equation} Suppose
$\overline{\lambda}_1,\ldots,\overline{\lambda}_n$ and
$\overline{\lambda}_1',\ldots,\overline{\lambda}_n'$ are the
eigenvalues of $\overline{\mathbf{A}}$ and
$n^{-1/2}\overline{\mathbf{A}}$, respectively. Clearly,
$\sum_{i=1}^n|\overline{\lambda}_i|=n^{1/2}\sum_{i=1}^n|\overline{\lambda}_i'|$.
By Eq.(\ref{Convergence}), we can deduce that
\begin{eqnarray*}
  \En\left(\overline{\mathbf{A}}\right)/n^{3/2} &=&
  \frac{1}{n^{3/2}}\sum_{i=1}^n|\overline{\lambda}_i|\\
  &=&
  \frac{1}{n}\sum_{i=1}^n|\overline{\lambda}_i'|\\
  &=& \int |x|d\Phi_{n^{-1/2}\overline{\mathbf{A}}}(x)\\
  &\rightarrow& \int
  |x|d\Phi(x)\mbox{ a.s. } (n\rightarrow\infty)\\
  &=&\frac{1}{2\pi\si_2^2}\int_{-2\si_2}^{2\si_2}
  |x|\sqrt{4\si_2^2-x^2}~dx\\
  &=&
  \frac{8}{3\pi}\si_2
  =\frac{8}{3\pi}\sqrt{p(1-p)}.
 \end{eqnarray*}
Therefore, the energy $\En\left(\overline{\mathbf{A}}\right)$
enjoys a.s. the equation as follows:
$$\En\left(\overline{\mathbf{A}}\right)
=n^{3/2}\left(\frac{8}{3\pi}\sqrt{p(1-p)}+o(1)\right).
$$
We proceed to investigating $\En(\mathbf{A})=\En(G_n(p))$ and
present the following result due to Fan.
\begin{lem}[\bf Fan \cite{ky}]\label{ky}
Let $\mathbf{X},\mathbf{Y},\mathbf{Z}$ be real symmetric matrices of
order $n$ such that $\mathbf{X}+\mathbf{Y}=\mathbf{Z}$. Then
$$\sum_{i=1}^n |\lambda_i(\mathbf{X})|+\sum_{i=1}^n
|\lambda_i(\mathbf{Y})| \geq\sum_{i=1}^n |\lambda_i(\mathbf{Z})|$$
where $\lambda_i(\mathbf{M})$ $(i=1,\cdots,n)$ are the eigenvalues
of the matrix $\mathbf{M}$.
\end{lem}
It is not difficult to verify that the eigenvalues of the matrix
$\mathbf{J}_n-\mathbf{I}_n$ are $n-1$ and $-1$ of $n-1$ times.
Consequently $\En(\mathbf{J}_n-\mathbf{I}_n)=2(n-1)$. One can
readily see that
$\En\big(p(\mathbf{J}_n-\mathbf{I}_n)\big)=p\En(\mathbf{J}_n-\mathbf{I}_n)$.
Thus,
$$\En\big(p(\mathbf{J}_n-\mathbf{I}_n)\big)=2p(n-1).$$ Since
$\mathbf{A}=\overline{\mathbf{A}}+p(\mathbf{J}_n-\mathbf{I}_n)$,
it follows from Lemma \ref{ky} that with probability 1,
\begin{eqnarray*}\En(\mathbf{A})
&\le&\En\left(\overline{\mathbf{A}}\right)+\En(p(\mathbf{J}_n-\mathbf{I}_n))\\
&=&n^{3/2}\left(\frac{8}{3\pi}\sqrt{p(1-p)}+o(1)\right)+2p(n-1).
\end{eqnarray*}Consequently,
\begin{equation}\label{Equ-4}\lim_{n\rightarrow\infty}\En(\mathbf{A})/n^{3/2}
\le\frac{8}{3\pi}\sqrt{p(1-p)}\mbox{ a.s.}
\end{equation}
On the other hand, since
$\overline{\mathbf{A}}=\mathbf{A}+p\big(-(\mathbf{J}_n-\mathbf{I}_n)\big)$,
we can deduce by Lemma \ref{ky} that with probability 1,
\begin{eqnarray*}\En(\mathbf{A})
&\ge&\En\left(\overline{\mathbf{A}}\right)
-\En\left(p\big(-(\mathbf{J}_n-\mathbf{I}_n)\big)\right)\\
&=&\En\left(\overline{\mathbf{A}}\right)-\En(p(\mathbf{J}_n-\mathbf{I}_n))\\
&=&n^{3/2}\left(\frac{8}{3\pi}\sqrt{p(1-p)}+o(1)\right)-2p(n-1).
\end{eqnarray*}
Consequently,
\begin{equation}\label{Equ-5}\lim_{n\rightarrow\infty}\En(\mathbf{A})/n^{3/2}
\ge\frac{8}{3\pi}\sqrt{p(1-p)}\mbox{ a.s.}\end{equation} Combining
Ineq.(\ref{Equ-4}) with Ineq.(\ref{Equ-5}), we have
$$\En(\mathbf{A})
=n^{3/2}\left(\frac{8}{3\pi}\sqrt{p(1-p)}+o(1)\right)\mbox{
a.s.}$$ Recalling that $\mathbf{A}$ is the adjacency matrix of
$G_n(p)$, we thus obtain that a.e. random graph $G_n(p)$ enjoys
the equation as follows:
$$\En(G_n(p))
=n^{3/2}\left(\frac{8}{3\pi}\sqrt{p(1-p)}+o(1)\right).
$$

\noindent{\bf Remark 2.2}. Note that for $p=\frac 1 2$, Nikiforov in
\cite{N} got the above equation. Here, our result is for any
probability $p$, which could be seen as a generalization of his
result.

\section{The energy of the random multipartite graph}

We begin with the definition of the random multipartite graph. We
use $K_{n;\nu_1,\ldots,\nu_m}$ to denote the complete $m$-partite
graph with vertex set $[n]$ whose parts $V_1,\ldots,V_m$ ($m=m(n)\ge
2$) are such that $|V_i|=n\nu_i=n\nu_i(n)$, $i=1,\ldots,m$. Let
$\mathcal{G}_{n;\nu_1\ldots\nu_m}(p)$ be the set of random
$m$-partite graphs with vertex set $[n]$ in which the edges are
chosen independently with probability $p$ from the set of edges of
$K_{n;\nu_1,\ldots,\nu_m}$. We further introduce two classes of
random $m$-partite graphs. Denote by $\mathcal{G}_{n,m}(p)$ and
$\mathcal{G}'_{n,m}(p)$, respectively, the sets of random
$m$-partite graphs satisfy, respectively, the following conditions:
\begin{equation}\label{Condition}
\lim_{n\rightarrow\infty}\max\{\nu_1(n),\ldots,\nu_m(n)\}>0\mbox{
and }
 \lim_{n\rightarrow\infty}\frac{\nu_i(n)}{\nu_j(n)}=1.
\end{equation} and
\begin{equation}\label{con2}
  \lim_{n\rightarrow\infty}\max\{\nu_1(n),\ldots,\nu_m(n)\}=0.
\end{equation}

One can easily see that to obtain the estimate of the energy for the
random  multipartite graph
$G_{n;\nu_1\ldots\nu_m}(p)\in\mathcal{G}_{n;\nu_1\ldots\nu_m}(p)$,
we need to investigate the spectral distribution of the random
matrix $\mathbf{A}(G_{n;\nu_1\ldots\nu_m}(p))$. It is not difficult
to verify that $\mathbf{A}(G_{n;\nu_1\ldots\nu_m}(p))$ would be a
special case of a random matrix $\mathbf{X}_n(\nu_1,\ldots,\nu_m)$
(or $\mathbf{X}_{n,m}$ for short) called a {\em random multipartite
matrix} which satisfies the following properties:\begin{itemize}

\item $x_{ij}$'s are independent random variables with $x_{ij}=x_{ji}$;

\item the $x_{ij}$'s have the same distribution $F_1$ if
$i$ and $j\in V_k$, while the $x_{ij}$'s have the same distribution
$F_2$ if $i\in V_k$ and $j\in [n]\setminus V_k$, where
$V_1,\ldots,V_m$ are the parts of $K_{n;\nu_1,\ldots,\nu_m}$ and $k$
is an integer with $1\le k\le m$;

\item $|x_{ij}|\le K$ for some constant $K$.
\end{itemize}
Apparently, if $F_1$ is a  pointmass at 0 and $F_2$ is a Bernoulli
distribution with mean $p$, then the random matrix
$\mathbf{X}_{n,m}$ coincides with the adjacency matrix
$\mathbf{A}(G_{n;\nu_1\ldots\nu_m}(p))$. Thus, we can readily
evaluate the energy $\En(G_{n;\nu_1\ldots\nu_m}(p))$ once we obtain
the spectral distribution of $\mathbf{X}_{n,m}$. In fact, the random
matrix $\mathbf{X}_{n,m}$ is a special case of the random matrix
considered by Anderson and Zeitouni \cite{AZ} in a rather general
setting called the band matrix model which can be regarded as one of
generalization of the Wigner matrix, and we shall employ their
results to deal with the spectral distribution of
$\mathbf{X}_{n,m}$.

The rest of this section will be divided into two parts. In the
first part, we shall present, respectively, exact estimates of the
energies for random graphs $G_{n,m}(p)\in\mathcal{G}_{n,m}(p)$ and
$G'_{n,m}(p)\in\mathcal{G}'_{n,m}(p)$ by exploring the spectral
distribution of the band matrix. We establish lower and upper bounds
of the energy for the random multipartite graph
$G_{n;\nu_1\ldots\nu_m}(p)$, and moreover we obtain an exact
estimate of the energy for the random bipartite graph
$G_{n;\nu_1,\nu_2}(p)$ in the second part.

\subsection{The energy of $G_{n,m}(p)$ and $G'_{n,m}(p)$}

In this part, we shall formulate exact estimates of the energies for
random graphs $G_{n,m}(p)$ and $G'_{n,m}(p)$, respectively. For this
purpose, we shall establish the following theorem. To state our
result, we first present some notations. Let
$\mathbf{I}_{n,m}=(i_{p,q})_{n\times n}$ be a {\em quasi-unit
matrix} such that
\begin{eqnarray*} i_{p,q}=\left\{
\begin{array}{ll} 1 &\mbox{ if }p,q\in V_k,\\
0 &\mbox{ if }p\in V_k\mbox{ and }q\in[n]\setminus
V_k,\end{array}\right.
\end{eqnarray*}
where $V_1,\ldots,V_m$ are the parts of $K_{n;\nu_1,\ldots,\nu_m}$
and $k$ is an integer with $1\le k\le m$. Set $\mu_i=\int
x\hspace{2pt}dF_i$ $(i=1,2)$ and
$$\overline{\mathbf{X}}_{n,m}
=\mathbf{X}_{n,m}-\mu_1\mathbf{I}_{n,m}-
\mu_2(\mathbf{J}_{n}-\mathbf{I}_{n,m}).
$$
Evidently, $\overline{\mathbf{X}}_{n,m}$ is a random multipartite
matrix as well in which each entry has mean 0. To make our statement
concise, we define $\Delta^2=(\si_1^2+(m-1)\si_2^2)/m$.
\begin{thm}\label{Main Thm}~
\begin{itemize}

\item[\rm (i)] If condition (\ref{Condition}) holds, then
$$\Phi_{n^{-1/2}\overline{\mathbf{X}}_{n,m}}(x)
\rightarrow_P\Psi(x)\mbox{ as }n\rightarrow\infty
$$
{\it i.e.,} the ESD $\Phi_{n^{-1/2}\overline{\mathbf{X}}_{n,m}}(x)$
converges weakly to a distribution $\Psi(x)$ in probability as $n$
tends to infinity where $\Psi(x)$ has the density
$$\psi(x)=
\frac{1}{2\pi\Delta^2} \sqrt{4\Delta^2-x^2}~\mathbf{1}_{|x|\le
2\Delta}.
$$

\item[\rm (ii)] If condition (\ref{con2}) holds, then
$\Phi_{n^{-1/2}\overline{\mathbf{X}}_{n,m}}(x) \rightarrow_P\Phi(x)$
as $n\rightarrow\infty$.
 \end{itemize}
 \end{thm}

Our theorem can be proved by a result established by Anderson and
Zeitouni \cite{AZ}. We begin with a succinct introduction of the
band matrix model defined by Anderson and Zeitouni in \cite{AZ},
from which one can readily see that a random multipartite matrix is
a band matrix.

We fix a non-empty set $\mathcal{C}=\{c_1,c_2,\ldots,c_m\}$ which is
finite or countably infinite. The elements of $\mathcal{C}$ are
called {\em colors}. Let $\kappa$ be a surjection from $[n]$ to the
color set $\mathcal{C}$, and we say that $\kappa(i)$ is the color of
$i$. Naturally, we can obtain a partition $V_1,\ldots,V_m$ of $[n]$
according to the colors of its elements, {\it i.e.,} two elements
$i$ and $i'$ in $[n]$ belong to the same part $V_j$ if and only if
their colors are identical. We next define the probability measure
$\theta_m$ on the color set as follows:
$$\theta_m(C)=\theta_{m(n)}(C)=|\kappa^{-1}(C)|/n, 1\le i\le m=m(n),$$
where $C\subseteq\mathcal{C}$ and
$\kappa^{-1}(C)=\{x\in[n]:\kappa(x)\in C\}$. Evidently, the
probability space $(\mathcal{C},2^{\mathcal{C}},\theta_m)$ is a
discrete probability space. Set
$$\theta=\lim_{n\rightarrow\infty}\theta_m.$$

For each positive integer $k$ we fix a bounded nonnegative function
$d^{(k)}$ on color set and a symmetric bounded nonnegative function
$s^{(k)}$ on the product of two copies of the color set. We make the
following assumptions:
\begin{itemize}

\item  $d^{(k)}$ is constant for $k\neq 2$;

\item $s^{(k)}$ is constant for $k\notin\{2, 4\}$;

\end{itemize}
Let $\{\xi_{ij}\}_{i,j=1}^n$ be a family of independent real-valued
mean zero random variables. We suppose that for all $1\le i,j\le n$
and positive integers $k$,
$$\label{Moments}
 \E(|\xi_{ij}|^k)\le \left\{\begin{array}{ll}
 s^{(k)}(\kappa(i),\kappa(j)) & \mbox{if }i\neq j,\\
 d^{(k)}(\kappa(i)) & \mbox{if }i= j,
 \end{array}\right.
$$
and moreover we assume that equality holds above whenever one of the
following conditions holds:
\begin{itemize}

\item $k = 2$,

\item $i\neq j$ and $k=4$.

\end{itemize}
In other words, the rule is to enforce equality whenever the
not-necessarily-constant functions $d^{(2)}, s^{(2)}$ or $s^{(4)}$
are involved, but otherwise merely to impose a bound.

We are now ready to present the random symmetric matrix
$\mathbf{Y}_n$ called {\it band matrix} in which the entries are
the r.v. $\xi_{ij}$. Evidently, $\mathbf{Y}_n$ is the same as
$\overline{\mathbf{X}}_{n,m}$ providing
\begin{equation}\label{Second moment}
s^{(2)}(\kappa(i),\kappa(j))= \left\{\begin{array}{ll} \si_1^2 &
\mbox{ if } \kappa(i)=\kappa(j)\\
\si_2^2 &\mbox{ if }
\kappa(i)\neq\kappa(j)\\
\end{array}\right. \mbox{ and } d^{(2)}(\kappa(i))=\si_1^2, 1\le i,j\le
n.
\end{equation}
So the random multipartite matrix $\overline{\mathbf{X}}_{n,m}$
is a special case of the band matrix $\mathbf{Y}_n$.

Define the standard semi-circle distribution $\Phi_{0,1}$ of zero
mean and unit variance to be the measure on the real set of compact
support with density
$\phi_{0,1}(x)=\frac{1}{2\pi}\sqrt{4-x^2}~\mathbf{1}_{|x|\le 2}$.
Anderson and Zeitouni investigated the LSD of $\mathbf{Y}_n$ and
proved the following result (Theorem 3.5 in \cite{AZ}).

\begin{lem}[\bf Anderson and Zeitouni \cite{AZ}]\label{Thm 3.5}
If $\int s^{(2)}(c,c')\theta(dc')\equiv 1$, then
$\Phi_{n^{-1/2}\mathbf{Y}_{n}}(x)$ converges weakly to the standard
semi-circle distribution $\Phi_{0,1}$ in probability as $n$ tends to
infinity. \end{lem}

\noindent{\bf Remark 3.1.} The main approach employed by Anderson
and Zeitouni to prove the assertion is a combinatorial enumeration
scheme for the different types of terms that contribute to the
expectation of products of traces of powers of the matrices. It is
worthwhile to point out that by an analogous method called moment
approach one can readily obtain a stronger assertion for
$\overline{\mathbf{X}}_{n,m}$ that the convergence could be valid
with probability 1. Moreover, one can show that for each positive
integer $k$,
\begin{equation}\label{Equ-6}
\lim_{n\rightarrow\infty} \int
x^k\Phi_{n^{-1/2}\overline{\mathbf{X}}_{n}}(x)=\left\{
\begin{array}{ll}
\displaystyle\int x^k\Psi(x)\mbox{ a.s.}&\mbox{ if condition
(\ref{Condition}) holds,}\\
\displaystyle\int x^k\Phi(x)\mbox{ a.s.}&\mbox{ if condition
(\ref{con2}) holds.}\end{array}\right.
\end{equation}
However, we shall not present the proof of Eq.(\ref{Equ-6}) here
since the arguments of the two methods are similar and the
calculation of the moment approach is rather tedious. We refer the
readers to Bai's survey \cite{B} for details.\vspace{6pt}

Using Lemma \ref{Thm 3.5}, to prove Theorem \ref{Main Thm}, we just
need to verify $\int s^{(2)}(c,c')\theta(dc')\equiv 1$. For Theorem
\ref{Main Thm}(i), we consider the matrix
$\Delta^{-1}\overline{\mathbf{X}}_{n,m}$ where
$\Delta^2=(\si_1^2+(m-1)\si_2^2)/m$. Note that condition
(\ref{Condition}) implies that $\theta_m(c_i)\rightarrow1/m$ as
$n\rightarrow\infty$, $1\le i\le m$. By means of condition
(\ref{Second moment}), one can readily see that for the random
matrix $\Delta^{-1}\overline{\mathbf{X}}_{n,m}$,
$$\int s^{(2)}(c,c')\theta(dc')=
\frac{1}{\Delta^2}\left(\frac{\si_1^2}{m}+
 \frac{(m-1)\si_2^2}{m}\right) \equiv1.$$
Consequently, Lemma \ref{Thm 3.5} implies that
$$\Phi_{n^{-1/2}\Delta^{-1}\overline{\mathbf{X}}_{n,m}}\rightarrow_P
\Phi_{0,1}\mbox{ as }n\rightarrow\infty.
$$
Therefore,
$$\Phi_{n^{-1/2}\overline{\mathbf{X}}_{n,m}}\rightarrow_P
\Psi(x)\mbox{ as }n\rightarrow\infty,
$$
and thus the first part
of Theorem \ref{Main Thm} follows.

For the second part of Theorem \ref{Main Thm}, we consider the
matrix $\si_2^{-1}\overline{\mathbf{X}}_{n,m}$. Note that condition
(\ref{con2}) implies that
$\theta(c_i)=\lim_{n\rightarrow\infty}\theta_m(c_i)=\lim_{n\rightarrow\infty}\nu_i(n)=0$,
$1\le i\le m$. By virtue of condition (\ref{Second moment}), if
$c\neq c'$ then $s^{(2)}(c,c')=1$. Consequently, for the random
matrix $\si_2^{-1}\overline{\mathbf{X}}_{n,m}$, we have
\begin{eqnarray*}
 \int s^{(2)}(c,c')\theta(dc')
 &=& \int
 s^{(2)}(c,c')\chi_{_{\mathcal{C}\setminus\{c\}}}\theta(dc')\\
 &=& \int \chi_{_{\mathcal{C}\setminus\{c\}}}\theta(dc')\\
 &=&\theta(\mathcal{C}\setminus\{c\})\equiv1.
\end{eqnarray*}
As a result, Lemma \ref{Thm 3.5} implies that
$$\Phi_{n^{-1/2}\sigma_2^{-1}\overline{\mathbf{X}}_{n,m}}\rightarrow_P
\Phi_{0,1}\mbox{ as }n\rightarrow\infty.$$ Therefore,
$$\Phi_{n^{-1/2}\overline{\mathbf{X}}_{n,m}}\rightarrow_P
\Phi(x)\mbox{ as }n\rightarrow\infty,$$ and thus the second part
follows.

We now employ Theorem \ref{Main Thm} to estimate the energy of
$G_{n;\nu_1\ldots\nu_m}(p)$ under condition (\ref{Condition}) or
(\ref{con2}). For convenience, we shall use $\mathbf{A}_{n,m}$ to
denote the adjacency matrix $\mathbf{A}(G_{n,m}(p))$. One can
readily see that if a random multipartite matrix $\mathbf{X}_{n,m}$
satisfies condition (\ref{Condition}), and $F_1$ is a pointmass at 0
and $F_2$ is a Bernoulli distribution with mean $p$, then
$\mathbf{X}_{n,m}$ coincides with the adjacency matrix
$\mathbf{A}_{n,m}$. Set
\begin{equation}\label{Equ-7}
\overline{\mathbf{A}}_{n,m}=\mathbf{A}_{n,m}
-p(\mathbf{J}_n-\mathbf{I}_{n,m})\end{equation} where
$\mathbf{I}_{n,m}$ is the quasi-unit matrix whose parts are the same
as $\mathbf{A}_{n,m}$. Evidently, each entry of
$\overline{\mathbf{A}}_{n,m}$ has mean 0. It follows from the first
part of Theorem \ref{Main Thm} that
$$\Phi_{n^{-1/2}\overline{\mathbf{A}}_{n,m}}\rightarrow_P
\Psi(x)\mbox{ as }n\rightarrow\infty.
$$
Since the density of $\Psi(x)$ is bounded with the finite support,
we can use a similar method for showing Eq.(\ref{Convergence}) to
prove that
$$\int |x|d\Phi_{n^{-1/2}\overline{\mathbf{A}}_{n,m}}(x)
\rightarrow_P\int |x|d\Psi(x)\mbox{ as } n\rightarrow\infty.$$
Consequently,
\begin{eqnarray*}
  \En\left(\overline{\mathbf{A}}_{n,m}\right)/n^{3/2}
  &=& \int |x|d\Phi_{n^{-1/2}\overline{\mathbf{A}}_{n,m}}(x)\\
  &\rightarrow_P& \int
  |x|d\Psi(x)\mbox{ as } n\rightarrow\infty\\
  &=&\frac{m}{2\pi(m-1)\si_2^2}
  \int^{2\sqrt{\frac{m-1}{m}}\si_2}_{-2\sqrt{\frac{m-1}{m}}\si_2}
    |x|\sqrt{4\frac{(m-1)\si_2^2}{m}-x^2}~dx\\
  &=&
  \frac{8}{3\pi}\sqrt{\frac{m-1}{m}}\si_2=
  \frac{8}{3\pi}\sqrt{\frac{m-1}{m}p(1-p)}.
 \end{eqnarray*} Therefore, a.e. random matrix
$\overline{\mathbf{A}}_{n,m}$ enjoys the equation as follows:
$$\En\left(\overline{\mathbf{A}}_{n,m}\right)
=n^{3/2}\left(\frac{8}{3\pi}\sqrt{\frac{m-1}{m}p(1-p)}+o(1)\right).$$

We now turn to the estimate of the energy
$\En(\mathbf{A}_{n,m})=\En(G_{n,m}(p))$. Evidently,
$$\mathbf{J}_n-\mathbf{I}_{n,m}=(\mathbf{J}_n-\mathbf{I}_n)
+(\mathbf{I}_{n}-\mathbf{I}_{n,m}).$$ By virtue of Lemma
\ref{ky}, we have
$$\En(\mathbf{J}_n-\mathbf{I}_{n,m})\le
\En(\mathbf{J}_n-\mathbf{I}_n)+\En(\mathbf{I}_{n}-\mathbf{I}_{n,m}).$$
Recalling the definition of the quasi-unit matrix $\mathbf{I}_{n,m}$
and the fact that $\En(\mathbf{J}_n-\mathbf{I}_n)=2(n-1)$, we have
$\En(\mathbf{J}_n-\mathbf{I}_{n,m})\le O(n).$ According to
Eq.(\ref{Equ-7}), we can use a similar argument for the estimate of
the energy $\En(\mathbf{A})$ from $\En(\overline{\mathbf{A}})$ to
show that a.e. random matrix $\mathbf{A}_{n,m}$ enjoys the equation
as follows:
$$\En(\mathbf{A}_{n,m})=
n^{3/2}\left(\frac{8}{3\pi}\sqrt{\frac{m-1}{m}p(1-p)}+o(1)\right).$$
Since the random matrix $\mathbf{A}_{n,m}$ is the adjacency matrix
of $G_{n,m}(p)$, we thus show that a.e. random graph $G_{n,m}(p)$
enjoys the following equation:
$$\En(G_{n,m}(p))= n^{3/2}\left(\frac{8}{3\pi}\sqrt{\frac{m-1}{m}p(1-p)}+o(1)\right).
$$

In what follows, we shall use $\mathbf{A'}_{n,m}$ to denote the
adjacency matrix $\mathbf{A}(G'_{n,m}(p))$. It is easily seen that
if a random multipartite matrix $\mathbf{X}_{n,m}$ satisfies
condition (\ref{con2}), and $F_1$ is a pointmass at 0 and $F_2$ is a
Bernoulli distribution with mean $p$, then $\mathbf{X}_{n,m}$
coincides with the adjacency matrix $\mathbf{A'}_{n,m}$. Set
\begin{equation*}\label{Equ-8}
\overline{\mathbf{A}'}_{n,m}=\mathbf{A}'_{n,m}
-p(\mathbf{J}_n-\mathbf{I}'_{n,m})\end{equation*} where
$\mathbf{I}'_{n,m}$ is the quasi-unit matrix whose parts are the
same as $\mathbf{A}'_{n,m}$. One can readily check that each entry
in $\overline{\mathbf{A}'}_{n,m}$ has mean 0. It follows from the
second part of Theorem \ref{Main Thm} that
$$\Phi_{n^{-1/2}\overline{\mathbf{A'}}_{n,m}}(x)\rightarrow_P\Phi(x)
\mbox{ as }n\rightarrow\infty.
$$
Employing the argument analogous to the estimate of
$\En(p(\mathbf{J}_n-\mathbf{I}_{n,m}))$,
$\En(\overline{\mathbf{A}}_{n,m})$ and $\En(\mathbf{A}_{n,m})$, one
can evaluate, respectively,
$\En(p(\mathbf{J}_n-\mathbf{I}'_{n,m}))$,
$\En(\overline{\mathbf{A}'}_{n,m})$ and $\En(\mathbf{A}'_{n,m})$,
and finally show that a.e. random graph $G'_{n,m}(p)$ satisfying
condition (\ref{con2}) enjoys the following equation:
\begin{equation}\label{a}
\En(G'_{n,m}(p))=n^{3/2}\left(\frac{8}{3\pi}\sqrt{p(1-p)}+o(1)\right).
\end{equation}

\subsection{The energy of $G_{n;\nu_1\ldots\nu_m}(p)$}

In this part, we shall give an estimate of energy for the random
multipartite graph $G_{n;\nu_1\ldots\nu_m}(p)$ satisfying the
following condition:
\begin{equation}\label{less}
\lim_{n\rightarrow\infty}\max\{\nu_1(n),\ldots,\nu_m(n)\}>0 \mbox{
and there exist $\nu_i$ and $\nu_j$ , }
\lim_{n\rightarrow\infty}\frac{\nu_i(n)}{\nu_j(n)}<1 .
\end{equation}
Moreover, for random bipartite graphs $G_{n;\nu_1,\nu_2}(p)$
satisfying $\lim_{n\rightarrow\infty}\nu_i(n)>0$ $(i=1,2)$, we shall
formulate an exact estimate of its energy.

Anderson and Zeitouni \cite{AZ} established the existence of the LSD
of $\mathbf{X}_{n,m}$ with partitions satisfying condition
(\ref{less}). Unfortunately, they failed to get the exact form of
the LSD, which appears to be much hard and complicated. However, we
can establish the lower and upper bounds for the energy
$\En(G_{n;\nu_1\ldots\nu_m}(p))$ via another way.

Here, we still denote the adjacency matrix of multipartite graph
satisfying condition (\ref{less}) by $\mathbf{A}_{n,m}$. Without
loss of generality, we assume, for some $r\ge 1$,
$|V_1|,\ldots,|V_r|$ are of order $O(n)$ while
$|V_{r+1}|,\cdots,|V_m|$ of order $o(n)$. Let $\mathbf{A}'_{n,m}$ be
a random symmetric matrix such that

$$\mathbf{A}'_{n,m}(ij)=\left\{\begin{array}{ll}\mathbf{A}_{n,m}(ij) &
\mbox{if  $i$ or $j\notin V_s, 1\leq s\leq r$,}\\
t_{ij} & \mbox{if $i,j\in V_s, 1\leq s\leq r$ and $i>j$,}\\
0 & \mbox{if $i,j\in V_s  ~(r+1\leq s\leq m)$ or $i=j,$}\\
\end{array}\right.
$$
where $t_{ij}$'s are independent Bernulli r.v. with mean $p$.
Evidently, $\mathbf{A}'_{n,m}$ is a random multipartite matrix. By
means of Eq.(\ref{a}), we have
$\En(\mathbf{A}'_{n,m})=\left(\frac{8}{3\pi}\sqrt{p(1-p)}+o(1)\right)n^{3/2}$.

Set
\begin{equation}\label{dxa}
\mathbf{D}_n=\mathbf{A}'_{n,m}-\mathbf{A}_{n,m}=
\left(\begin{array}{lllllll} \mathbf{K}_1 & & & &\\
 & \mathbf{K}_2& & & \\
 & &\ddots& &\\
 & & & \mathbf{K}_r&\\
 & & & & \mathbf{0}\\
 & & & & &\ddots&\\
 & & & & & &\mathbf{0}
\end{array}\right)_{n\times n}\end{equation}
One can readily see that $\mathbf{K}_i~(i=1,\ldots,r)$ is a Wigner
matrix and thus a.e. $\mathbf{K}_i$ enjoys the following:
$$\En(\mathbf{K}_i)=\left(\frac{8}{3\pi}\sqrt{p(1-p)}+o(1)\right)(\nu_in)^{3/2}.$$
Consequently, a.e. matrix $\mathbf{D}_n$ satisfies the following:
$$\En(\mathbf{D}_n)=\left(\frac{8}{3\pi}\sqrt{p(1-p)}+o(1)\right)\left(\nu_1^{\frac{3}{2}}
+\cdots+\nu_r^{\frac{3}{2}}\right)n^{\frac{3}{2}}.
$$
By Eq.(\ref{dxa}), we have
$\mathbf{A}_{n,m}+\mathbf{D}_n=\mathbf{A}'_{n,m}$ and
$\mathbf{A}'_{n,m}+(-\mathbf{D}_n)=\mathbf{A}_{n,m}$. Employing
Lemma \ref{ky}, we deduce
$$\En(\mathbf{A}'_{n,m})-\En(\mathbf{D}_n)\leq \En(\mathbf{A}_{n,m})
\leq \En(\mathbf{A}'_{n,m})+\En(\mathbf{D}_n).
$$
Recalling that $\mathbf{A}_{n,m}$ is the adjacency matrix of
$G_{n;\nu_1\ldots\nu_m}(p)$, the following result is relevant.
\begin{thm}\label{budengbu}
Almost every random graph $G_{n;\nu_1\ldots\nu_m}(p)$ satisfies the
inequality below
$$\left(1-\sum_{i=1}^r\nu_i^{\frac{3}{2}}\right)n^{3/2}
\le
\En(G_{n;\nu_1\ldots\nu_m}(p))\left(\frac{8}{3\pi}\sqrt{p(1-p)}+o(1)\right)^{-1}
\le\left(1+\sum_{i=1}^r\nu_i^{\frac{3}{2}}\right)n^{3/2} .$$
\end{thm}

\noindent\textbf{Remark 3.2}. Since $\nu_1,\ldots,\nu_r$ are
positive real numbers with $\sum_{i=1}^r\nu_i\le 1$, we have
\linebreak $\sum_{i=1}^r\nu_i(1-\nu_i^{1/2})>0$. Therefore,
$\sum_{i=1}^r\nu_i>\sum_{i=1}^r\nu_i^{3/2},$ and thus
$1>\sum_{i=1}^r\nu_i^{3/2}$. Hence, we can deduce, by the theorem
above, that a.e. random graph $G_{n;\nu_1\ldots\nu_m}(p)$ enjoys the
following $$\En(G_{n;\nu_1\ldots\nu_m}(p))=O(n^{3/2}).$$

In what follows, we investigate the energy of random bipartite
graphs $G_{n;\nu_1,\nu_2}(p)$ satisfying
$\lim_{n\rightarrow\infty}\nu_i(n)>0$ $(i=1,2)$, and present the
precise estimate of $\En(G_{n;\nu_1,\nu_2}(p))$ via
Mar\v{c}enko-Pastur Law.

For convenience, set $n_1=\nu_1n$ and $n_2=\nu_2n$. Let
$\mathbf{I}_{n,2}$ be a quasi-unit matrix with the same partition as
$\mathbf{A}_{n,2}$. Set
\begin{equation}\label{z}
\overline{\mathbf{A}}_{n,2}=\mathbf{A}_{n,2}-p(\mathbf{J}_n-\mathbf{I}_{n,2})=\left[\begin{array}{ll}
\mathbf{O}& \mathbf{X}^{T}\\
\mathbf{X}& \mathbf{O}\end{array}\right],\end{equation} where
$\mathbf{X}$ is a random matrix of order $n_2\times n_1$ in which
the entries $\mathbf{X}(ij)$ are iid. with mean zero and variance
$\si^2=p(1-p)$. By the equation
$$
\left(\begin{array}{ll} \lambda\mathbf{I}_{n_1} & \mathbf{0} \\
  -\mathbf{X} & \lambda\mathbf{I}_{n_2}
\end{array}\right)
\left(\begin{array}{ll} \lambda\mathbf{I}_{n_1} &
-\mathbf{X}^T \\
\mathbf{0} &
\lambda\mathbf{I}_{n_2}-\lambda^{-1}\mathbf{X}\mathbf{X}^T
\end{array}\right)=\lambda
\left(\begin{array}{ll} \lambda\mathbf{I}_{n_1}&
-\mathbf{X}^T \\
-\mathbf{X} & \lambda\mathbf{I}_{n_2}
\end{array}\right),$$
we have
$$\lambda^{n}\cdot\lambda^{n_1}|\lambda\mathbf{I}_{n_2}-\lambda^{-1}\mathbf{X}\mathbf{X}^{T}|=\lambda^{n}|\lambda
\mathbf{I}_n-\overline{\mathbf{A}}_{n,2}|,
$$
and consequently,
$$\lambda^{n_1}|\lambda^2\mathbf{I}_{n_2}-\mathbf{X}\mathbf{X}^{T}|=\lambda^{n_2}|\lambda
\mathbf{I}_n-\overline{\mathbf{A}}_{n,2}|.
$$
Thus, the eigenvalues of $\overline{\mathbf{A}}_{n,2}$ are
symmetric, and moreover $\overline{\lambda}$ is the eigenvalue of
$\frac{1}{\sqrt{n_1}}\overline{\mathbf{A}}_{n,2}$ if and only if
$\overline{\lambda}^2$ is the eigenvalue of
$\frac{1}{n_1}\mathbf{X}\mathbf{X}^{T}$. Therefore, we can
characterize the spectral of $\overline{\mathbf{A}}_{n,2}$ by the
spectral of $\mathbf{X}\mathbf{X}^T$. Bai formulated the LSD of
$\frac{1}{n_1}\mathbf{X}\mathbf{X}^T$ (Theorem 2.5 in \cite{B}) by
moment approach.
\begin{lem}[\bf  Mar\v{c}enko-Pastur Law \cite{B}]\label{mp}
Suppose that $\mathbf{X}(ij)$'s are iid. with mean zero and variance
$\si^2=p(1-p)$, and $\nu_2/\nu_1\rightarrow y\in (0,\infty)$. Then,
with probability 1, the ESD
$\Phi_{\frac{1}{n_1}\mathbf{X}\mathbf{X}^T}$ converges weakly to the
Mar\v{c}enko-Pastur Law $F_y$ as $n\rightarrow\infty$ where $F_y$
has the density
$$f_y(x)=
\frac{1}{2\pi p(1-p)x y}\sqrt{(b-x)(x-a)}~\mathbf{1}_{a\leq x\leq b}
$$
and has a point mass $1-1/y$ at the origin if $y>1$ where
$a=p(1-p)(1-\sqrt{y})^2$ and $b=p(1-p)(1+\sqrt{y})^2$.
\end{lem}

By the symmetry of the eigenvalues of
$\frac{1}{\sqrt{n_1}}\overline{\mathbf{A}}_{n,2}$, to evaluate the
energy $\En(\frac{1}{\sqrt{n_1}}\overline{\mathbf{A}}_{n,2})$, we
just need to consider the positive eigenvalues. Define
$\Theta_{n_2}(x)=\frac{\sum1_{{\overline{\lambda}<x }}}{n_2}$. One
can see that the sum of the positive eigenvalues of
$\frac{1}{\sqrt{n_1}}\overline{\mathbf{A}}_{n,2}$ equals
$n_2\int_{0}^{\infty} xd\Theta_{n_2}(x)$. Suppose $0<x_1<x_2$, we
have
$$\Theta_{n_2}(x_2)-\Theta_{n_2}(x_1)=\Phi_{\frac{1}{n_1}\mathbf{X}\mathbf{X}^T}(x_2^2)
-\Phi_{\frac{1}{n_1}\mathbf{X}\mathbf{X}^T}(x_1^2).
$$
It follows that
$$\int_{0}^{\infty} xd\Theta_{n_2}(x)=\int_{0}^{\infty}
\sqrt{x}d\Phi_{\frac{1}{n_1}\mathbf{X}\mathbf{X}^T}(x).
$$
Note that all the eigenvalues of
$\frac{1}{n_1}\mathbf{X}\mathbf{X}^T$ are nonnegative. By the moment
approach (see \cite{B} for instance), we have
\begin{eqnarray*}\int
x^2d\Phi_{\frac{1}{n_1}\mathbf{X}\mathbf{X}^{T}}(x)
&=&\int_{0}^{\infty}
x^2d\Phi_{\frac{1}{n_1}\mathbf{X}\mathbf{X}^{T}}(x)\\
&\rightarrow& \int_{0}^{\infty} x^2dF_y(x)\mbox{ a.s. }
(n\rightarrow\infty)\\
&=&\int x^2dF_y(x)
\end{eqnarray*}
Analogous to the proof of Eq.(\ref{Convergence}), we can deduce that
$$
\lim_{n\rightarrow\infty}\int_{0}^{\infty}
\sqrt{x}d\Phi_{\frac{1}{n_1}\mathbf{X}\mathbf{X}^{T}}(x)=
\displaystyle\int_{0}^{\infty} \sqrt{x}dF_y(x)  \mbox{ a.s. }$$
Therefore,
$$\lim_{n\rightarrow\infty}\int_{0}^{\infty}
xd\Theta_{n_2}(x)=\int_{\sqrt{a}}^{\sqrt{b}}\frac{1}{\pi
p(1-p)y}\sqrt{(b-x^2)(x^2-a)}~dx\mbox{ a.s. }
$$
Let
$$\Lambda=\int_{\sqrt{a}}^{\sqrt{b}}\frac{1}{\pi
p(1-p)y}\sqrt{(b-x^2)(x^2-a)}~dx.
$$
We obtain that for a.e. $\overline{\mathbf{A}}_{n,2}$ the sum of the
positive eigenvalues is $(\Lambda+o(1))n_2\sqrt{n_1}$. Thus, a.e.
$\En(\overline{\mathbf{A}}_{n,2})$ enjoys the equation as follows:
$$\En(\overline{\mathbf{A}}_{n,2})=(2\Lambda+o(1))n_2\sqrt{n_1}.$$
Furthermore, we can get
$$\Lambda=\frac{\sqrt{b}[(a+b)\Ep(1-a/b)-2a\Ek(1-a/b)]}{3\pi
p(1-p)y},
$$
where $\Ek$ is the complete elliptic integral of the first kind and
$\Ep$ is the complete elliptic integral of the second kind. Let
$t\in[0,1]$, the two kinds of complete elliptic integral are defined
as follows
$$\Ek(t)=\int_0^{\frac{\pi}{2}}\frac{d\theta}{\sqrt{1-t\sin^2\theta}}
\mbox{  and }
\Ep(t)=\int_0^{\frac{\pi}{2}}\sqrt{1-t\sin^2\theta}~d\theta.
$$
The value can be got by mathematical software for every parameter
$t$.

Employing Eq.(\ref{z}) and Lemma \ref{ky}, we have
$$\En(\overline{\mathbf{A}}_{n,2})-\En(p(\mathbf{J}_n-\mathbf{I}_{n,2}))\leq
\En({\mathbf{A}_{n,2}})\leq\En({\overline{\mathbf{A}}_{n,2}})+\En(p(\mathbf{J}_n-\mathbf{I}_{n,2})).$$
Together with the fact that
$\En(p(\mathbf{J}_n-\mathbf{I}_{n,2}))=2p\sqrt{\nu_1\nu_2}n$ and
$n_2\sqrt{n_1}=\nu_2\sqrt{\nu_1}n^{3/2}$, we get
$$\En({\mathbf{A}_{n,2}})=(2\nu_2\sqrt{\nu_1}\Lambda+o(1))n^{3/2}.$$
Therefore, the following theorem is relevant.

\begin{thm}\label{bi}
Almost every random bipartite graph $G_{n;\nu_1,\nu_2}(p)$ with
$\nu_2/\nu_1\rightarrow y$ enjoys
$$\mathscr{E}(G_{n;\nu_1,\nu_2}(p))=(2\nu_2\sqrt{\nu_1}\Lambda+o(1))n^{3/2}.$$
\end{thm}

We now compare the above estimate of the energy
$\mathscr{E}(G_{n;\nu_1,\nu_2}(p))$ with bounds obtained in Theorem
\ref{budengbu} for $p=1/2$. In fact, Koolen and Moulton \cite{Ko}
established the following upper bound of the energy $\En(G)$ for
simple graphs $G$:
$$\En(G)\le\frac{n}{2}(\sqrt{n}+1).$$
Consequently, for $p=1/2$, this upper bound is better than ours. So
we turn our attention to compare the estimate of
$\En(G_{n;\nu_1,\nu_2}(1/2))$ in Theorem \ref{bi} with the lower
bound in Theorem \ref{budengbu}. By the numerical computation (see
the table below), the energy $\En(G_{n;\nu_1,\nu_2}(1/2))$ of a.e.
random bipartite $G_{n;\nu_1,\nu_2}(1/2)$ is close to our lower
bound.

\begin{center}\begin{tabular}{|c|c|c|}\hline

$y$ & $\En(G_{n;\nu_1,\nu_2}(p))$ & lower bound of
$\En(G_{n;\nu_1,\nu_2}(p))$\\\hline

1 & $(0.3001+o(1))n^{3/2}$ & $(0.1243+o(1))n^{3/2}$\\\hline

2 & $(0.2539+o(1))n^{3/2}$ & $(0.1118+o(1))n^{3/2}$\\\hline

3 & $(0.2071+o(1))n^{3/2}$ & $(0.0957+o(1))n^{3/2}$\\\hline

4 & $(0.1731+o(1))n^{3/2}$ & $(0.0828+o(1))n^{3/2}$\\\hline

5 & $(0.1482+o(1))n^{3/2}$ & $(0.0727+o(1))n^{3/2}$\\\hline

6 & $(0.1294+o(1))n^{3/2}$ & $(0.06470+o(1))n^{3/2}$\\\hline

7 & $(0.1148+o(1))n^{3/2}$ & $(0.05828+o(1))n^{3/2}$\\\hline

8 & $(0.1031+o(1))n^{3/2}$ & $(0.05301+o(1))n^{3/2}$\\\hline

9 & $(0.09353+o(1))n^{3/2}$ & $(0.04862+o(1))n^{3/2}$\\\hline

10 & $(0.08558+o(1))n^{3/2}$ & $(0.04491+o(1))n^{3/2}$\\\hline
\end{tabular}\end{center}

\end{document}